\documentclass[a4paper,12pt,oneside]{article}

\usepackage{amsmath}
\usepackage{amssymb}
\usepackage{amsthm}
\usepackage{hyperref}

\newtheorem{example}{Example}[section]

\newtheorem{theorem}[example]{Theorem}
\newtheorem{corollary}[example]{Corollary}

\usepackage{dsfont}

\title{What is the probability that two elements of a finite ring have product zero?} 
\author{Sanhan M. S. Khasraw\\
Department of Mathematics, College of Education,\\ Salahaddin University-Erbil, Erbil, Iraq\\
sanhan.khasraw@su.edu.krd
}
\date{}

\begin{document}
\maketitle

\begin{center}
{\large{\textbf{Abstract}}}\\
\end{center}
In this paper we consider the probability that two elements of a finite ring have product zero. We find bounds of this probability of a finite commutative ring with identity 1. The explicit computations for the ring $\mathds{Z}_n$, the ring of integers modulo $n$, have been obtained. \\

\textbf{Keywords}: Zero-divisor, zero-divisor graph, probability.

\section{Introduction}

The problem of finding the probability $P(G)$ that two elements of a finite group $G$ commute was considered by Gustafson \cite{MR02944871}. He showed that $P(G)\leq 5/8$. For more studies about probability and group theory, see\cite{rusin1979probability, lescot1988degre}.

I. Beck in 1988 \cite{beck1988coloring} introduced the notion of \textit{zero-divisor graphs} of a commutative ring. The zero-divisor graph whose vertices are non-zero zero-divisors in which two vertices are adjacent if their product is zero. Zero-divisor graph of a commutative ring has been studied by many authors, see \cite{anderson2003zero, anderson1999zero, livingston1997structure}.

In this paper the probability $P(R)$ that two elements of a finite ring $R$ multiply to 0 is considered. The probability $P(R)$ that two elements chosen at random (with replacement) from a ring $R$ have product zero is 
$P(R)=\frac{|Ann|}{|R \times R|}$,
where $Ann=\{(x, y)\in R\times R \;|\; xy=0\}$.\\
 To count the elements of $Ann$, one can see that for each $x \in R$, the number of elements of $Ann$ of the form $(x, y)$ is $|Ann(x)|$, where $Ann(x)$ is the annihilator of $x$ in $R$. Hence
      $|Ann|=\sum_{x \in R} |Ann(x)|,$
where the sum is taken over all $x\in R$.
Note that if $xy=0$, then both $(x,y)$ and $(y,x)$ are elements of $Ann$. 

Throughout this paper, all rings are assumed to be finite commutative with identity 1 in
order to establish lower and upper bounds of $P(R)$.
Furthermore, the author assumes that the identity element 1 is different from 0, as the zero ring is a finite commutative ring with identity (namely $0$ since $0 \cdot 0=0$) and the corresponding probability that two elements multiply to $0$ is $1$.
Also, in case of a ring $R$ does not have an identity element, the probability that two elements multiply to $0$ can be $1$, for instance, if $R$ is a ring such that $ab=0$ for all $a, b \in R$, then $P(R)=1$. These two cases are exempted throughout this paper in order to investigate minimum and maximum values of the probability $P(R)$.

\section{Bounds for P(R)}
In this section, the general lower and upper bounds for $P(R)$ will be found.

\begin{theorem}\label{lower}
Suppose $|R|=l$. Then $P(R)\geq \frac{2l+|Z(R)|-1}{l^2}$, where $Z(R)$ is the set of nonzero zero-divisors of $R$.
\end{theorem}
\begin{proof}
It is clear that $|Ann(0)|=l$. Suppose $Z(R)$ be the set of nonzero zero-divisors of $R$. For every $x \in Z(R)$ we have that $|Ann(x)| \geq 2$,  and $|Ann(x)|=1$ for each $0 \neq x \notin Z(R)$. Thus, 
$$|Ann|=|Ann(0)|+\sum_{x \in Z(R)} |Ann(x)|+\sum_{0 \neq x \notin Z(R)} |Ann(x)|$$\\
$$\geq l+2 \cdot |Z(R)|+(l-1-|Z(R)|)\cdot 1=2l+|Z(R)|-1.$$
Therefore, 
$P(R)=\frac{|Ann|}{l^2}\geq \frac{2l+|Z(R)|-1}{l^2}$.
\end{proof}

\medskip

\begin{theorem}\label{upper}
Suppose $|R|=l$. Then $P(R)\leq \frac{2l+(m-1)|Z(R)|-1}{l^2}$,  where $Z(R)$ is the set of nonzero zero-divisors of $R$, and $m=max\{|Ann(x)| \;:\; x\in Z(R)\}$.
\end{theorem}
\begin{proof}
Suppose $|R|=l$. Again, $|Ann(0)|=l$, and let the number of nonzero zero-divisors of $R$ is $k$, that is, $k:=|Z(R)|$. 
Suppose that $m:=max\{|Ann(x)| \;:\; x\in Z(R)\}$. Note that $k$ and $m$ are vary while $l$ varies. 
 Thus,

$$|Ann|=|Ann(0)|+\sum_{x \in Z(R)} |Ann(x)|+\sum_{0 \neq x \notin Z(R)} |Ann(x)|$$\\
$$\leq l+m \cdot k+(l-1-k)\cdot 1=2l+(m-1)k-1.$$\\
So, 
$$P(R)\leq \frac{2l+(m-1)k-1}{l^2}.$$
\end{proof}

\medskip

\begin{corollary}
$P(R)\leq \frac{3}{4}$.
\end{corollary}
\begin{proof}
Since $Ann(x)$ is an ideal of $R$ for any $x\in R$, it must be the case that $m=max\{|Ann(x)| \;:\; x\in Z(R)\}\leq \frac{1}{2}l$ and, in general, $k=|Z(R)|\leq l-2$, for then, by Theorem \ref{upper}, 
$P(R)\leq \frac{2l+(\frac{1}{2}l-1)(l-2)-1}{l^2}=\frac{1}{2}+\frac{1}{l^2}$ which is decreasing according to $l$. In particular, if $l=2$, then $P(R)\leq \frac{3}{4}$. 
\end{proof}

From Theorem \ref{lower} and Theorem \ref{upper}, the following corollaries can be stated.

\begin{corollary}\label{cor1}
If $R$ is an integral domain and $|R|=l$, then $P(R)=\frac{2l-1}{l^2}$.
\end{corollary}
\begin{proof}
Since $R$ is an integral domain, then $|Z(R)|=0$. Thus, $\frac{2l-1}{l^2}\leq P(R)\leq \frac{2l-1}{l^2}$. The result follows.
\end{proof}

For the non-integral domain case, if all annihilators have size $m$, then the inequalities of Theorem \ref{lower} and Theorem \ref{upper} lead to the following.

\begin{corollary}\label{cor36}
If $R$ is a non-integral domain, $|R|=l$ and $|Ann(x)|=m$ for all $x \in Z(R)$, then $P(R)=\frac{2l+(m-1)k-1}{l^2}$.
\end{corollary}

\medskip
Before we start a new section, we require to have the following theorem of computing the probability of direct product of rings.

\begin{theorem}\label{pro}
Let $R=R_1 \times R_2$, where $R_1$ and $R_2$ are rings. Then $P(R)=P(R_1)P(R_2)$.
\end{theorem} 
\begin{proof}
Recall that $P(R)=\frac{|Ann|}{|R \times R|}$, where
 $Ann=\{((r_1, s_1), (r_2, s_2))\in R\times R \;|\; (r_1r_2, s_1s_2)=(0, 0)\}$.
 Rewrite $Ann$ in terms of $Ann(R_1)$ and $Ann(R_2)$, where 
 $Ann(R_1):=\{(r_1, r_2)\in R_1\times R_1 \;|\; r_1r_2=0\}$ and   
 $Ann(R_2):=\{(s_1, s_2)\in R_2\times R_2 \;|\; s_1s_2=0\}$, as follows: \\
 $Ann=\{((r_1, s_1), (r_2, s_2))\in R\times R \;|\; (r_1r_2, s_1s_2)=(0, 0)\}=$
$\{(r_1, r_2)\in R_1\times R_1 \;|\; r_1r_2=0\}$$\{(s_1, s_2)\in R_2\times R_2 \;|\; s_1s_2=0\}=$
$Ann(R_1)$$Ann(R_2)$. Thus, $P(R)=\frac{|Ann|}{|R \times R|}=\frac{|Ann(R_1)||Ann(R_2)|}{|R_1 \times R_1||R_2 \times R_2|}=P(R_1)P(R_2)$ 
\end{proof} 

\section{The ring $\mathds{Z}_n$}

In this section $P(\mathds{Z}_n)$ will be found, where $\mathds{Z}_n$ is the ring of integers modulo $n$. It is well known that if $n=p_1^{k_1} \cdot p_2^{k_2} \cdots p_r^{k_r}$, then $\mathds{Z}_n \cong \mathds{Z}_{p_1^{k_1}} \times \mathds{Z}_{p_2^{k_2}} \times \cdots \times \mathds{Z}_{p_r^{k_r}}$. From Theorem \ref{pro}, we only require to find $P(\mathds{Z}_{p^k})$, for some positive integer $k$.

\begin{theorem}\label{thm}
$P(\mathds{Z}_{p^k})=\frac{(k+1)p-k}{p^{k+1}}$, where $p$ is a prime and $k \geq 1$.
\end{theorem}
\begin{proof}
It is clear that the set of nonzero zero divisors of $\mathds{Z}_{p^k}$ is $ S:= \{ p,$ $ 2p,$ $ 3p,$ $ \cdots, $ $(p^{k-1} - 1) p \} $ with size $p^{k-1}-1$. Rewrite the set $S$ as the union of $k-1$ disjoint sets $S_i$, $i=1, 2, \cdots, k-1$, that is, $S=\bigcup_{i=1}^{k-1}S_i$, such that $S_i$ contains $mp^i$, $m \neq lp$, the multiple of $p$, and $m=1, 2, \cdots, p^{k-i}-1$. Thus, the size of each $S_i$ is $(p^{k-i}-1)-(p^{k-(i+1)}-1)=p^{k-i}-p^{k-(i+1)}$. One can see that the product of any element of $S_i$ with every element of $\bigcup_{j=1}^{i}S_{k-j}$ is zero. So, the annihilator of each element of $S_i$ has size $p^i$. Hence
$|Ann|=|Ann(0)|+\sum_{x \in S} |Ann(x)|+\sum_{0 \neq x \notin S} |Ann(x)|=p^k+\sum_{i=1}^{k-1}p^i(p^{k-i}-p^{k-(i+1)})+(p^k-1-(p^{k-1}-1))=p^k+\sum_{i=1}^{k-1}(p^k-p^{k-1})+(p^k-p^{k-1})=(k+1)p^k-kp^{k-1}$. Therefore, $P(\mathds{Z}_{p^k})=\frac{|Ann|}{(p^k)^2}=\frac{p^{k-1}((k+1)p-k)}{p^{2k}}=\frac{(k+1)p-k}{p^{k+1}}$.
\end{proof}

\begin{corollary}
If  $n=p_1^{k_1} \cdot p_2^{k_2} \cdots p_r^{k_r}$, then $P(\mathds{Z}_n)=\prod_{i=1}^{r}\frac{(k_i+1)p_i-k_i}{p^{k_i+1}}$.
\end{corollary}
\begin{proof}
It is well known that if $n=p_1^{k_1} \cdot p_2^{k_2} \cdots p_r^{k_r}$, then $\mathds{Z}_n \cong \mathds{Z}_{p_1^{k_1}} \times \mathds{Z}_{p_2^{k_2}} \times \cdots \times \mathds{Z}_{p_r^{k_r}}$. By Theorem \ref{pro}, $P(\mathds{Z}_n)=\prod_{i=1}^{r}P(\mathds{Z}_{p_i^{k_i}})$. From Theorem \ref{thm}, the result follows. 
\end{proof}

\end{document}